\documentclass[11pt]{article}
\usepackage[T1]{fontenc}
\usepackage[utf8]{inputenc}
\usepackage{geometry}
\usepackage{amsmath}
\usepackage{amsthm}
\usepackage{setspace}
\usepackage[unicode=true,
 bookmarks=false,
 breaklinks=false,pdfborder={0 0 1},colorlinks=false]
 {hyperref}

\makeatletter

\theoremstyle{plain}
\newtheorem{thm}{\protect\theoremname}
\theoremstyle{plain}
\newtheorem{lem}[thm]{\protect\lemmaname}
\theoremstyle{plain}
\newtheorem{claim}[thm]{\protect\claimname}
\theoremstyle{plain}
\newtheorem{cor}[thm]{\protect\corname}

\usepackage{amssymb}

\makeatother

\providecommand{\claimname}{Claim}
\providecommand{\corname}{Corollary}
\providecommand{\lemmaname}{Lemma}
\providecommand{\theoremname}{Theorem}
\newcommand{\cupdot}{\mathbin{\mathaccent\cdot\cup}}

\begin{document}
\title{A multipartite analogue of Dilworth's Theorem}
\author{Jacob Fox\thanks{Department of Mathematics, Stanford University, Stanford, CA 94305. Email: {\tt jacobfox@stanford.edu}. Research supported by a Packard Fellowship and by NSF Awards DMS-1855635 and DMS-2154169.} \and Huy Tuan Pham\thanks{Department of Mathematics, Stanford University, Stanford, CA 94305. Email: {\tt huypham@stanford.edu}. Research supported by a Two Sigma Fellowship, a Clay Research Fellowship and a Stanford Science Fellowship.}}
\maketitle
\begin{abstract}
We prove that every partially ordered set on $n$ elements 
contains $k$ subsets $A_{1},A_{2},\dots,A_{k}$ such that either each of these subsets has size $\Omega(n/k^{5})$ and, for every $i<j$, every element in $A_{i}$ is less than or equal to every element in $A_{j}$, or each of these subsets has size $\Omega(n/(k^{2}\log n))$
and, for every $i \not = j$, every element in $A_{i}$ is incomparable with every element in $A_{j}$ for $i\ne j$. This answers a question of the first author from 2006. As a corollary, we prove for each positive integer $h$ there is $C_h$ such that for any $h$ partial orders $<_{1},<_{2},\dots,<_{h}$ on a set of $n$ elements, there exists $k$ subsets $A_{1},A_{2},\dots,A_{k}$ each of size at least $n/(k\log n)^{C_{h}}$ such that for each partial order $<_{\ell}$, either $a_{1}<_{\ell}a_{2}<_{\ell}\dots<_{\ell}a_{k}$ for any tuple of elements $(a_1,a_2,\dots,a_k) \in A_1\times A_2\times \dots \times A_k$, or $a_{1}>_{\ell}a_{2}>_{\ell}\dots>_{\ell}a_{k}$ for any $(a_1,a_2,\dots,a_k) \in A_1\times A_2\times \dots \times A_k$, or $a_i$ is incomparable with $a_j$ for any $i\ne j$, $a_i\in A_i$ and $a_j\in A_j$. This improves on a 2009 result of Pach and the first author motivated by problems in discrete geometry. 
\end{abstract}

\section{Introduction}

In a partially ordered set, a \textit{chain} is a set of pairwise comparable elements, and an \textit{antichain} is a set of pairwise incomparable elements. Dilworth's theorem \cite{D} implies that for all positive integers $k$ and $n$, every partially ordered set with $n$ elements contains a chain of size $k$ or an antichain of size $\lceil n/k\rceil$. In particular, every partially ordered set on n elements contains a chain or antichain of size $\lceil\sqrt{n}\rceil$. Equivalently, every comparability graph on $n$ vertices contains a clique or independent set of size at least $\sqrt{n}$. In \cite{F}, the first
author shows that one can guarantee a much larger balanced complete bipartite
subgraph in a comparability graph or its complement. It is further shown in \cite{FPT} that there exists a constant $c>0$ such that in every partially ordered set of size $n\ge 2$, there are disjoint subsets $A,B$ of size at least $cn$ where every
element of $A$ is larger than every element of $B$, or there are
disjoint subsets $A,B$ of size at least $\frac{cn}{\log n}$ where
every element of $A$ is incomparable with any element of $B$. For
subsets $A,B$ of $P$, we write $A<B$ if $a<b$ for all $a\in A,b\in B$. We also write $B>A$ if $A<B$. We say that $A$ and $B$ are \textit{comparable} if $A<B$ or $B<A$.
We say that $A$ and $B$ are \textit{totally incomparable} if $a,b$ are
incomparable for all $a\in A$ and $b\in B$. Let $m_k(n)$ be the largest integer such that every partially ordered set of $n$ elements contains 
$k$ disjoint subsets $A_1,\dots,A_k$ each of size $m_k(n)$ such that either $A_{1}>A_{2}>\dots>A_{k}$ or all pairs of different
subsets $A_{i},A_{j}$ are totally incomparable. It is proved in \cite{F} that $m_{2}(n)=\Theta\left(n/\log n\right)$. For $k\geq 2$, we clearly have  $m_{k}(n)\le m_{2}(n)=O(n/\log n)$. By iterating the lower bound on $m_2(n)$, we obtain the lower bound $m_{2^{k}}(n)\ge\frac{c^k n}{(\log n)^{2k-1}}$ for a positive constant $c$ and $n$ sufficiently large in terms of $k$.  The first author asked in \cite{F} if this lower bound can be improved, and in particular,
whether $m_{3}(n)=\Omega\left(\frac{n}{\log n}\right)$. Our main result shows that this is indeed the case. 
\begin{thm}
\label{thm:multi-dilworth}Let $k\ge 2$ and  $n\ge (100k)^5$ be integers and
$(P,<)$ be a partially ordered set on $n$ elements. Then
there exists $k$ disjoint subsets $A_{1},\dots,A_{k}$ of $P$ such that either each $A_i$
 has size at least $10^{-4}k^{-5}n$ and they satisfy $A_{1}>A_{2}>\dots>A_{k}$,
or each $A_i$ has size at least $40^{-1}k^{-2}n/\log n$ and 
$A_{1},\dots,A_{k}$ are pairwise totally incomparable. 
\end{thm}

In fact, we prove a more general version of Theorem \ref{thm:multi-dilworth}, Theorem \ref{thm:multi-dilworth-general}, which allows for a more general trade-off between the lower bound in the comparable and totally incomparable case. For example, we also obtain the following result. 
\begin{thm}
\label{thm:multi-dilworth-variant}Let $k\ge 2$ and  $n\ge 10^{10} k^3 (\log k)^2$ be integers and 
$(P,<)$ be a partially ordered set on $n$ elements. Then there exists $k$ disjoint subsets $A_{1},\dots,A_{k}$ of $P$ such that either each $A_i$ has size at least $10^{-4}k^{-3}(\log k)^{-2}n$ and $A_{1}>A_{2}>\dots>A_{k}$,
or each $A_i$ has size at least $20^{-1}k^{-2}(\log k)^{-1}n/\log n$ and 
$A_{1},\dots,A_{k}$ are pairwise totally incomparable. 
\end{thm}

Theorem \ref{thm:multi-dilworth} implies $m_{k}(n)\ge\frac{n}{40k^{2}\log n}$ for $n$ sufficiently large in terms of $k$, so the lower bound on $m_{k}(n)$ matches the trivial upper bound $m_{k}(n) \le \frac{Cn}{\log n}$ up to a constant factor depending on $k$. In fact, the dependency on $k$ in Theorem \ref{thm:multi-dilworth} is also best possible, which we next show. In \cite{F}, it is shown that for any $0<\epsilon<1$, there exists a partially
ordered set $(Q,<)$ on $n$ elements such that any element of $Q$ is comparable with at most $n^{\epsilon}$ other elements, and there does not exist two disjoint subsets of $Q$ of size at least $\frac{14n}{\epsilon\log_{2}n}$ which are totally incomparable.
Consider the partially ordered set $P$ consisting of $k-1$ copies
of $Q$, labeled $Q_1,\dots,Q_{k-1}$, and $Q_1<Q_2<\dots<Q_{k-1}$. The partially ordered set $P$ has $(k-1)n$ elements. 
There does not exist $k$ disjoint subsets $P_1<P_2<\dots<P_k$ of $P$ each of size at
least $n^{\epsilon}$. On the other hand, any pair of disjoint subsets of $P$ which are totally incomparable
must both be subsets of the same copy $Q_i$ of $Q$, hence, each
such set has size at most $\frac{14n}{\epsilon\log_{2}n}$. From any
$k \geq 2$ pairwise totally incomparable sets each of size $s$, we can find two
totally incomparable sets each of size at least $s\lfloor k/2 \rfloor \geq sk/3$. These sets are the union of $\lfloor k/2 \rfloor$ sets and the union of the remaining $\lceil k/2 \rceil$ sets of the $k$ pairwise totally incomparable sets. Hence, $sk/3 < \frac{14n}{\epsilon\log_{2}n}$. Taking $\epsilon=1/2$, we have $s<\frac{84n}{k\log_{2}n}<\frac{200|P|}{k^2\log_{2}|P|}$.
This shows that the result in Theorem \ref{thm:multi-dilworth} is
best possible in terms of the dependency on $n$ and $k$ in the incomparable case.
It also determines $m_{k}(n)$ up to an absolute constant. 
\begin{cor}
\label{cor:m_k} Let $k\ge 2$ be an integer. For $n\ge (100k)^5$, we have $$m_{k}(n) = \Theta\left(\frac{n}{k^2\log n}\right).$$
\end{cor}

The initial study of these extremal problems for partially ordered sets was motivated by applications and connections to problems in discrete geometry. 

A family of graphs is said to have the \textit{Erd\H{o}s-Hajnal property} if there exists a constant $c>0$ such that each graph $G$ in the family contains either a clique or an independent set of size at least $|V(G)|^c$. The celebrated conjecture of Erd\H{o}s and Hajnal \cite{EH} states that any hereditary family of graphs which is not the family of all graphs has the Erd\H{o}s-Hajnal property. Given a family of geometric objects, the \textit{intersection graph} of a set of objects in the family is the graph whose vertices correspond to objects in the set, and two vertices are joined by an edge if the corresponding objects have a nonempty intersection. In the geometric setting, the Erd\H{o}s-Hajnal property of many families of intersection graphs of geometric objects can be understood from sets equipped with multiple partial orders. For example, Larman et al.~\cite{LMPT} and Pach and T\"{o}r\H{o}csik \cite{PT} introduced four partial orders on the set of convex sets in the plane with the property that two convex sets are disjoint if and only if they are comparable in at least one of the four partial orders. Using the four partial orders and Dilworth's Theorem, it can be shown that any set of $n$ convex sets in the plane contains a collection of $n^{1/5}$ pairwise intersecting sets or $n^{1/5}$ pairwise disjoint sets. Thus, the family of intersection graphs of convex sets in the plane has the Erd\H{o}s-Hajnal property. This result can be generalized to the family of intersection graphs of \textit{vertically convex sets}. A \textit{vertically convex set} is a compact connected
set in the plane with the property that any straight line parallel
to the $y$-axis of the coordinate system intersects it in an interval
(which may be empty or may consist of one point). In particular, any
$x$-monotone arc, that is, the graph of any continuous function defined
on an interval of the $x$-axis, is vertically convex. The four partial
orders introduced in \cite{LMPT,PT} satisfy the stronger property
that any two vertically convex sets are disjoint if and only if they
are comparable in at least one of the four partial orders. Thus, any
collection of $n$ vertically convex sets contain $n^{1/5}$ pairwise
intersecting sets or $n^{1/5}$ pairwise disjoint sets.

A family of graphs has the \textit{strong Erd\H{o}s-Hajnal property} if there exists a constant $c>0$ such that each graph $G$ in the family contains two subsets each of size at least $cn$ which are either empty or complete to each other. It is known that the strong Erd\H{o}s-Hajnal property implies the Erd\H{o}s-Hajnal property (see \cite{APPRS,FP08,FPS23}). In \cite{FP}, the first author and Pach study complete bipartite subgraphs in the comparability or incomparability graphs of multiple partial orders. In particular, using the bound $m_{2^{k}}(n)\ge\frac{c^{k}n}{(\log n)^{2k-1}}$, it is shown in \cite{FP} that given any $h$ partial orderings on the same set $P$ of size $n$, there exists two subsets $A$ and $B$ of $P$ of size at least $n2^{-(1+o(1))(\log\log n)^{h}}$ such that $A$ and $B$ are either comparable or totally incomparable in each of the $h$ partial orderings. This implies that any collection of $n$ vertically convex sets contains two subcollections $A$ and $B$ each of size at least
$n2^{-(1+o(1))(\log\log n)^{4}}$ such that any set in $A$ intersects
any set in $B$ or any set in $A$ is disjoint from any set in $B$. We remark that a construction of Kyn\v{c}l \cite{K}, improving on a previous result of K\'{a}rolyi et al.~\cite{KPT}, shows that for infinitely many $n$, there exists a family of $n$ line segments in the plane with at most $n^{.405}$ members that are pairwise intersecting or pairwise disjoint. Thus, we can guarantee much larger complete or empty bipartite graphs than cliques or independent set in the intersection graphs of vertically convex sets.

Our next result gives an improvement of the main result of \cite{FP}. Using Theorem \ref{thm:multi-dilworth}, we generalize the result of \cite{FP} to find $k$ sets which are pairwise comparable or pairwise totally incomparable in each of the partial orderings, and we also obtain an improved bound on the size of the sets even in the case $k=2$. 
\begin{thm}
\label{thm:multiple}Let $k$ be a positive integer. For $i\in[h]$,
let $(P,<_{i})$ be partial orderings on a set $P$ of size $n$, where $n$ is sufficiently large in terms of $k$. 
Then there exists $k$ sets $A_{1},\dots,A_{k}$ such that for each
$i\in[h]$, either $A_{1}<_{i}A_{2}<_{i}\dots<_{i}A_{k}$, or $A_{1}>_{i}A_{2}>_{i}\dots>_{i}A_{k}$,
or $A_{j},A_{j'}$ are totally incomparable in $<_{i}$ for all $j\ne j'$.
Furthermore, for all $j\le k$, $|A_{j}|\ge\frac{n}{(10k\log n)^{12^{h+1}}}$. 
\end{thm}
In \cite{FP}, a random construction is used to show that there exists
partial orderings $<_{1},\dots,<_{h}$ on $n$ elements such that
for any two sets $A,B$ which are either comparable or totally incomparable
in each of the partial orderings, we have $\min(|A|,|B|)\le\frac{C_hn(\log\log n)^{h-1}}{(\log n)^{h}}$, where $C_h$ is a positive constant depending only on $h$.
This shows that the polynomial dependency on $\log n$ in Theorem \ref{thm:multiple} is necessary.  An improved construction,  which removes the $\log \log n$ factors, was obtained by Kor\'andi and Tomon \cite{KT}. 

Theorem \ref{thm:multiple} leads to an improvement of a result of Kor\'andi, Pach and Tomon \cite{KPT2} on the existence of large homogeneous submatrices in a zero-one matrix avoiding certain $2\times k$ pattern. We say that a matrix $P$ is acyclic if every submatrix of $P$ has a row or column with at most one $1$. We say that $P$ is simple if $P$ is acyclic, and its complement $P^c$ obtained by switching $0$'s and $1$'s in $P$ is also acyclic. Kor\'andi, Pach and Tomon show that for any simple matrix $P$ of size $2\times k$, if an $n\times n$ matrix $A$ does not contain $P$ as a submatrix, then $A$ contains a submatrix of size $n2^{-(1+o(1))(\log \log n)^{k}} \times \Omega_k(n)$ whose entries are either all $0$ or all $1$. Their bound relies directly on a version of Theorem \ref{thm:multiple} for $h=2$, and Theorem \ref{thm:multiple} immediately leads to an improved bound of $\frac{n}{(10k\log n)^{1728}} \times \Omega_k(n)$ for the homogeneous submatrix. 

We next turn to discuss other geometric applications of Theorem \ref{thm:multiple}. An immediate corollary of Theorem \ref{thm:multiple} is that any collection of $n$ vertically convex sets contains two subcollections
$A$ and $B$ each of size at least $n/(\log n)^{10^6}$ such that
any set in $A$ intersects any set in $B$ or any set in $A$ is disjoint
from any set in $B$. 
Using other techniques, a stronger version of this corollary has already been shown in \cite{FPT}: any collection of $n$ convex sets in the plane contains two
subcollections $A$ and $B$ each of size at least $\Omega(n)$ such that
any set in $A$ intersects any set in $B$ or any set in $A$ is disjoint
from any set in $B$; and any collection of $n$ vertically convex
sets in the plane contains two subcollections $A$ and $B$ each of
size at least $\Omega\left(\frac{n}{\log n}\right)$ such that any set in $A$ intersects
any set in $B$ or any set in $A$ is disjoint from any set in $B$. 

The results on intersection graphs of vertically convex sets have also been generalized to \textit{string graphs}. A {\it string graph} is a graph whose vertices are curves in the plane, and two vertices are adjacent if and only if the two corresponding curves intersect. Using the result of \cite{FP2} showing that dense string graphs contain dense incomparability graphs, as well as the separator theorem for sparse string graphs \cite{L}, one can show that any string graph on $n$ vertices contains two subsets each of size at least $\Omega\left(\frac{n}{\log n}\right)$ which are empty or complete to each other. This result is tight up to the absolute constant, due to the construction of \cite{F}, and the observation that incomparability graphs are string graphs. In a recent breakthrough, Tomon \cite{T} shows that the family of string graphs has the Erd\H{o}s-Hajnal property. In another direction, Scott, Seymour and Spirkl \cite{SSS}, resolving a conjecture of the first author, shows that any perfect graph has two subsets of size at least $n^{1-o(1)}$ which are complete or empty to each other, generalizing the result of \cite{F} for incomparability graphs. 

We think that Theorem \ref{thm:multiple} should have further geometric applications. We also note that the proofs of our theorems easily give efficient polynomial time algorithms for finding the desired complete multipartite structures in partially ordered sets. In particular, these results should be applicable to some problems in computational geometry. We leave as an open problem the determination of  the optimal dependency on $k$ and $h$ in Theorem \ref{thm:multiple}. 

All logarithms are base $e$ unless otherwise specified.  For the sake of clarity of presentation, we sometimes often floor and ceiling signs whenever they are not crucial. 

\section{A multipartite analogue of Dilworth's Theorem}

Given a partially ordered set $(P,<)$ and an element $x\in P$, we
define $D_{P}(x)$ to be the set of elements $y\in P$ such that $y<x$,
and $U_{P}(x)$ the set of elements $y\in P$ such that $y>x$. We will prove the following more general result, from which Theorem \ref{thm:multi-dilworth} and Theorem \ref{thm:multi-dilworth-variant} easily follow. 
\begin{thm} \label{thm:multi-dilworth-general}
Let $f:\mathbb{Z}_+ \to \mathbb{R}_+$ be an increasing function such that $f(2)\ge 16$, and for all positive integers $k\ge 2$, we have $f(k) > 2f(\lfloor k/2\rfloor) + 6$ and $f(k) \ge 2f(\lceil k/2\rceil)$. Let $g:\mathbb{Z}_+\to \mathbb{R}_+$ be a decreasing function such that $g(2) \le \frac{1}{2}$, and for all $k\ge 2$, 
\[
g(k) \le \frac{\frac{1}{2}f(k)-f(\lfloor k/2\rfloor) - 3}{2k}. 
\]
Let $k\ge 2$ be a positive integer and let
$(P,<)$ be a partially ordered set on $n$ elements. Assume that $g(k)^2 n \ge 10^5 (kf(k)^2)$, then either
there exists $k$ disjoint subsets $A_{1},\dots,A_{k}$ of $P$ each
of size at least $\frac{1}{37} \frac{g(k)^2 n}{kf(k)^2}$ such that $A_{1}>A_{2}>\dots>A_{k}$,
or there exists $k$ disjoint subsets $A_{1},\dots,A_{k}$ of $P$
each of size at least $\frac{7n}{16kf(k)(\log n)}$ such that
$A_{1},\dots,A_{k}$ are pairwise totally incomparable. 
\end{thm}

First, we show that Theorem \ref{thm:multi-dilworth} and Theorem \ref{thm:multi-dilworth-variant} follow from Theorem \ref{thm:multi-dilworth-general}. 
\begin{proof}[Proof of Theorem \ref{thm:multi-dilworth} and Theorem \ref{thm:multi-dilworth-variant} assuming Theorem \ref{thm:multi-dilworth-general}]
To prove Theorem \ref{thm:multi-dilworth}, we apply Theorem \ref{thm:multi-dilworth-general} with $f(k)=16(k-1)$ and $g(k) = \frac{1}{k}$ for all positive integers $k$. We verify that these choices of $f$ and $g$ satisfy the conditions in Theorem \ref{thm:multi-dilworth-general}. Clearly $f$ is increasing, $g$ is decreasing, and $f(2)\ge 16$. For all positive integers $k \ge 2$, 
\begin{align*}
2 f(\lceil k/2\rceil) &= 2 \cdot 16\left(\left\lceil\frac{k}{2}\right\rceil - 1\right) \\
&\le 2\cdot 16\left(\frac{k+1}{2}-1\right)\gamma \\ 
&= 16(k-1) \\
&= f(k),
\end{align*}
and 
\begin{align*}
\frac{\frac{1}{2}f(k)-f(\lfloor k/2\rfloor) - 3}{2k} &= \frac{8(k-1) - 16(\lfloor k/2\rfloor -1) - 3}{2k}\\
&\ge \frac{8(k-1)-16(k/2-1) - 3}{2k} \\
&\ge \frac{1}{k} \\
&= g(k).
\end{align*}
This also implies that $f(k) > 2f(\lfloor k/2\rfloor)+6$ for all $k\ge 2$. 

Similarly, Theorem \ref{thm:multi-dilworth-variant} follows from Theorem \ref{thm:multi-dilworth-general} upon choosing $f(k) = 8 k \log k$ and $g(k) = 1/2$, which can similarly be shown to satisfy the conditions in Theorem \ref{thm:multi-dilworth-general}. 
\end{proof}

Next, we prove Theorem \ref{thm:multi-dilworth-general}, showing that any partially
ordered set either contains $k$ large sets that are pairwise totally incomparable
or $k$ large sets that are pairwise comparable. 

Theorem \ref{thm:multi-dilworth-general} follows from the following two lemmas.  

The first lemma shows that if a poset does not contain a chain $A_1>\cdots >A_k$ of $k$ large sets,  then its comparability graph contains a large induced subgraph which is quite sparse.  Tomon \cite{T16} and Pach,  Rubin,  and Tardos (Lemma 2.1,  \cite{PRT}) prove results which show that if a poset does not contain a chain of $k$ large sets,  then its comparability graph cannot be very dense. 
\begin{lem}
\label{lem:comp}Let $k$ be an integer and let $(P,<)$ a partially
ordered set. Let $\ell<|P|/k$ be a positive integer.
Then either there exists $k$ disjoint subsets $A_{1},\dots,A_{k}$
of $P$ each of size at least $\ell$ such that $A_{1}<A_{2}<\dots<A_{k}$,
or there exists a subset $Q$ of $P$ such that $|Q|\ge\frac{7|P|}{16k}$
and $|D_{Q}(x)|<4\sqrt{|Q|\ell}$ for all $x\in Q$, or
there exists a subset $Q$ of $P$ such that $|Q|\ge\frac{7|P|}{16k}$
and $|U_{Q}(x)|<4\sqrt{|Q|\ell}$ for all $x\in Q$. 
\end{lem}

\begin{lem}
\label{lem:incomp}Let $k\ge2$ be a positive integer and let $(Q,<)$
a partially ordered set. Let $\lambda\ge 0$ and $\gamma\ge 1$ be so that $\max_{x\in Q}|D_{Q}(x)|\le\lambda$, $\gamma\le\frac{|Q|}{f(k)}$ and $\lambda \le g(k)\gamma$. Then there exists $k$ disjoint subsets $A_{1},A_{2},\dots,A_{k}$ of $Q$ which are pairwise totally 
incomparable and each has size at least $\frac{\gamma}{k\log|Q|}$. 
\end{lem}

We next give the proof of Theorem \ref{thm:multi-dilworth-general} assuming Lemma \ref{lem:comp} and Lemma \ref{lem:incomp}. 

\begin{proof}[Proof of Theorem \ref{thm:multi-dilworth-general}]
Choose $\ell=\left \lceil \frac{1}{37} \frac{g(k)^2 n}{kf(k)^2} \right \rceil$. By Lemma \ref{lem:comp},
we can either find $k$ sets $A_{1},\dots,A_{k}$ each of size at
least $\ell$ such that $A_{1}>A_{2}>\dots>A_{k}$, or we can find
a subset $Q$ of $P$ with $|Q|\ge\frac{7n}{16k}$ and $|D_{Q}(x)|\le 4\sqrt{|Q|\ell}$
for all $x\in Q$, or we can find
a subset $Q$ of $P$ with $|Q|\ge\frac{7n}{16k}$ and $|U_{Q}(x)|\le 4\sqrt{|Q|\ell}$
for all $x\in Q$. 

In the first case, the conclusion of Theorem \ref{thm:multi-dilworth} holds. We next consider the second case where we can find
a subset $Q$ of $P$ with $|Q|\ge\frac{7n}{16k}$ and $|D_{Q}(x)|\le 4\sqrt{|Q|\ell}$ for all $x\in Q$. The third case can be treated similarly.

Note that $$4\sqrt{|Q|\ell} = 4|Q|\sqrt{\frac{\ell}{|Q|}} \le 4|Q|\sqrt{\frac{\ell}{7n/(16k)}}\le g(k)|Q|/f(k),$$ by the choice of $\ell$ and the assumption $g(k)^2 n \ge 10^5 kf(k)^2$. By Lemma \ref{lem:incomp} with $\gamma = |Q|/f(k) \ge 1$ and $\lambda = g(k)\frac{|Q|}{f(k)} = g(k)\gamma$, we can find $k$ pairwise totally
incomparable subsets of $Q$ each of size at least $\frac{|Q|/f(k)}{k\log |Q|}\ge\frac{7n}{16kf(k)(\log n)}$. Thus, the conclusion of Theorem \ref{thm:multi-dilworth} also holds in this case.
\end{proof}

We next prove Lemma \ref{lem:comp}.
\begin{proof}[Proof of Lemma \ref{lem:comp}]
Define the partial ordering $<_{\ell}$ on $P$ such that $x<_{\ell}y$
if and only if there exists $\ell$ distinct elements $a_{1},\dots,a_{\ell}\in P$
such that $x<a_{j}<y$ for all $j\in[\ell]$. Assume that there exists
a chain $x_{1}<_{\ell}x_{2}<_{\ell}\dots<_{\ell}x_{k+1}$ of length
$k+1$ in $(P,<_{\ell})$. Then there are distinct elements $a_{i,j}$
for $i\in[k]$, $j\in[\ell]$ such that $x_{i}<a_{i,j}<x_{i+1}$ for
all $i\in[k],j\in[\ell]$. The subsets $A_{i}=\{a_{i,j},j\in[\ell]\}$
then satisfy $A_{1}<A_{2}<\dots<A_{k}$. Thus, if there exists a chain
of length $k+1$ in $(P,<_{\ell})$, then we obtain $k$ subsets $A_{1},\dots,A_{k}$
of $P$ each of size at least $\ell$ such that $A_{1}<A_{2}<\dots<A_{k}$.

Otherwise, there does not exist a $(k+1)$-chain in $(P,<_{\ell})$,
so by Mirsky's Theorem (the dual of Dilworth's Theorem), there exists
a partition of $P$ to at most $k$ antichains in $(P,<_{\ell})$.
Thus, there exists an antichain in $(P,<_{\ell})$ of size at least
$\frac{|P|}{k}$. Let $P'$ be the elements of this antichain and let
$n'=|P'|$. For any $x,y\in P'$, there are less than $\ell$ elements
$z\in P'$ such that $x<z<y$. Thus the number of triples $(x,y,z)\in P'^{3}$
such that $x<z<y$ is at most $\binom{n'}{2}\ell<\frac{n'^{2}\ell}{2}$.
On the other hand, the number of triples $(x,y,z)\in P'^{3}$ such
that $x<z<y$ is equal to $\sum_{z\in P'}|D_{P'}(z)|\cdot|U_{P'}(z)|$.
Thus the number of elements $x\in P'$ with $\min(|D_{P'}(x)|,|U_{P'}(x)|)\ge2\sqrt{n'\ell}$
is at most $n'/8$. Hence, either at least $7n'/16$ elements $x\in P'$
satisfies $|D_{P'}(x)|<2\sqrt{n'\ell}$, or at least $7n'/16$ elements
$x\in P'$ satisfies $|U_{P'}(x)|<2\sqrt{n'\ell}$. Without loss of
generality, assume that at least $7n'/16$ elements $x\in P'$ satisfies
$|D_{P'}(x)|<2\sqrt{n'\ell}$. Let $Q$ be the set of elements $x\in P'$
with $|D_{P'}(x)|<2\sqrt{n'\ell}$. Then $Q$ has size at least $7n'/16$
and for all $x\in Q$, $|D_{Q}(x)|\le|D_{P'}(x)|<2\sqrt{n'\ell}<4\sqrt{|Q|\ell}$.
\end{proof}

For a subset $S$ of a partially ordered set $(Q,<)$, we denote by $D_{Q}(S)$ the set of elements $x\notin S$ such that there exists $s\in S$
with $s>x$. Given a set $S$ and a positive integer $k$, an {\it equitable partition} of $S$ into $k$ parts is a partition of $S$ into $k$ disjoint subsets each of size $\lfloor |S|/k\rfloor$ or $\lceil |S|/k\rceil$. We prove Lemma \ref{lem:incomp} by considering the following algorithms. For both algorithms below, we fix positive parameters $\gamma$ and $\lambda$ that the algorithms rely on which do not change throughout the execution of the algorithms. In the inputs to Algorithm 1, $B$ is a subset of a partially ordered set $Q$.

$\hfill$

\noindent $\textrm{Algorithm 1: Condense}(Q,B,k)$. 
\begin{enumerate}
\item If $|B| < k$, output $k$ empty sets. 
\item Otherwise, $|B| \ge k$. Consider an arbitrary equitable partition of $B$ into $k$ sets $B_{1},\dots,B_{k}$. 
\item If $|D_{Q}(B_{i})\setminus D_{Q}(B\setminus B_{i})|\ge\gamma/(k\log|Q|)$
for all $i=1,2,\dots,k$, output the sets $D_{Q}(B_{i})\setminus D_{Q}(B\setminus B_{i})$
for $i=1,2,\dots,k$.
\item Otherwise, there exists $i\in\{1,2,\dots,k\}$ such that $|D_{Q}(B_{i})\setminus D_{Q}(B\setminus B_{i})|<\gamma/(k\log|Q|)$. Let $i$ be the smallest such index and output $\textrm{Condense}(Q,B\setminus B_{i},k)$. 
\end{enumerate}

$\hfill$

\noindent $\textrm{Algorithm 2: Select}(Q,k)$. 
\begin{enumerate}
\item If $k=1$, output $Q$.
\item If $k>1$, consider an arbitrary linear extension of $Q$. Let $T$
be the top $\lceil|Q|/2\rceil$ elements in the linear ordering. 
\begin{itemize}
\item If $|D_{Q}(T)|\ge 2(k\lambda+\gamma)$, run $\textrm{Condense}(Q,T,k)$ and return its output.
\item If $|D_{Q}(T)| < 2(k\lambda+\gamma)$, run $\textrm{Select}(T,\lceil k/2\rceil)$
and $\textrm{Select}(Q\setminus(T\cup D_Q(T)),\lfloor k/2\rfloor)$
and return $k=\lceil k/2\rceil+\lfloor k/2\rfloor$ sets from the
two outputs. 
\end{itemize}
\end{enumerate}

Observe that in Algorithm 1 ($\textrm{Condense}(Q,B,k)$), if the algorithm terminates in Step 3, then the sets $D_Q(B_i)\setminus D_Q(B\setminus B_i)$ for $i=1,2,\dots,k$ are pairwise totally incomparable, since if $x_i \in D_Q(B_i)\setminus D_Q(B\setminus B_i)$, $x_j \in D_Q(B_j)\setminus D_Q(B\setminus B_j)$ for $i\ne j$ and $x_i < x_j$, then $x_i \in D_Q(B_j)$, which contradicts the fact that $x_i \notin D_Q(B\setminus B_i)$. In Claim \ref{claim:condense} below, we show that under appropriate conditions, $\textrm{Condense}(Q,B,k)$ always terminates in Step 3 and outputs $k$ totally incomparable sets of large size. Claim \ref{claim:select} below verifies that Algorithm 2 ($\textrm{Select}(Q,k)$) outputs $k$ large pairwise totally incomparable sets. In particular, in Algorithm 2, if we are in the first case of Step 2, then we output $k$ large pairwise totally incomparable sets by Claim \ref{claim:condense}. Otherwise, $|D_Q(T)|$ is small, so the sets $T$ and $Q\setminus (T\cup D_Q(T))$ are both large, and we obtain the desired outputs from the recursive calls to Algorithm 2. 

We now state and prove Claim \ref{claim:condense} and Claim \ref{claim:select}, from which the proof of Lemma \ref{lem:incomp} follows. 
\begin{claim}
\label{claim:condense}If $|D_{Q}(x)|\le\lambda$ for all
$x\in Q$, and $B\subseteq Q$ satisfies $|B|\ge k$ and $|D_{Q}(B)|\ge2(k\lambda+\gamma)$, then $\textrm{Condense}(Q,B,k)$ outputs $k$ totally incomparable
sets each of size at least $\gamma/(k\log|Q|)$. 
\end{claim}

\begin{proof}
Observe that in one iteration of $\textrm{Condense}(Q,B,k)$, either the algorithm stops in Step 1 or 3 and outputs $k$ disjoint and pairwise comparable sets, or it recursively calls and outputs $\textrm{Condense}(Q,B',k)$, where $B'$ is a particular subset of $B$. If the algorithm has not stopped and outputted $k$ sets after $j$ recursive calls, then it outputs $\textrm{Condense}(Q,B^{(j)},k)$ for some subset $B^{(j)}$ of $B$. Specifically, let $B=B^{(0)}$ and, for $j\ge 0$, if $\textrm{Condense}(Q,B^{(j)},k)$ does not stop in Step 1 or 3, then in Step 4 it outputs $\textrm{Condense}(Q,B^{(j+1)},k)$, where $B^{(j+1)}=B^{(j)}\setminus B_{i}^{(j)}$ and $B_{i}^{(j)}$ is one set in an equipartition of $B^{j}$ into $k$ sets such that 
\begin{equation}
|D_{Q}(B^{(j)})\setminus D_{Q}(B^{(j)}\setminus B_{i}^{(j)})|<\gamma/(k\log|Q|). \label{eq:boundD}
\end{equation}Note that if $|B^{(j)}|\ge k$, then 
\begin{equation}
|B_{i}^{(j)}|\ge\lfloor|B^{(j)}|/k\rfloor \ge |B^{(j)}|/(2k), \label{eq:boundlow}
\end{equation}and
\begin{equation}
|B_{i}^{(j)}| \le \lceil |B^{(j)}|/k\rceil.\label{eq:boundup}
\end{equation}
In particular, Inequality (\ref{eq:boundlow}) implies that $\textrm{Condense}(Q,B,k)$ must terminate at some number of recursive calls $t\ge 0$. 

Clearly, if $\textrm{Condense}(Q,B^{(t)},k)$ terminates in Step 3, then we output $k$ totally incomparable sets each of size at least $\gamma/(k\log|Q|)$, which completes the proof of the claim in this case. Assume for the sake of contradiction that $\textrm{Condense}(Q,B^{(t)},k)$ terminates but not in Step 3, so $\textrm{Condense}(Q,B^{(t)},k)$ terminates in Step 1 and hence $|B^{(t)}|<k$. 

As $|B^{(0)}|=|B| \geq k > |B^{(t)}|$ and as (\ref{eq:boundlow}) and (\ref{eq:boundup}) hold for $0 \leq j<t$, there is a smallest integer $s$ such that $k \leq |B^{(s)}| \leq 2k$. Here, we note that $x - \lceil x/k\rceil \ge k$ for all $x\ge 2k+1$ and $k\ge 2$. For $j\le s$, we have
 $\textrm{Condense}(Q,B^{(j)},k)$ does not terminate
in Step 3, $|B^{(j+1)}|\le\left(1-\frac{1}{2k}\right)|B^{(j)}|$ by (\ref{eq:boundlow}),
and $|D_{Q}(B^{(j)})|-|D_{Q}(B^{(j+1)})|\le\gamma/(k\log|Q|)$ by (\ref{eq:boundD}). In particular, we have $k\le |B^{(s)}|\le\left(1-\frac{1}{2k}\right)^{s}|B^{(0)}|\le \left(1-\frac{1}{2k}\right)^{s}|Q|$, so $s\le 2k\log |Q|$. Thus, we obtain
\[
|D_{Q}(B^{(s)})|>2(k\lambda+\gamma)-s\gamma/(k\log|Q|)\ge2(k\lambda+\gamma)-2\gamma = 2k\lambda.
\]
However, $|D_{Q}(B^{(s)})|\le2k\lambda$ since $|D_{Q}(x)|\le\lambda$
for all $x\in Q$ and $|B^{(s)}| \le 2k$. This contradiction shows that $\textrm{Condense}(Q,B^{(j)},k)$ must terminate in Step 3 for some $j\in [0,t]$ and output $k$ totally incomparable sets each of size at least $\gamma/(k\log|Q|)$. 
\end{proof}

\begin{claim}
\label{claim:select}Let $k\ge 2$ be a positive integer. If $|Q| \ge f(k)\gamma$, and $|D_{Q}(x)|\le\lambda\le g(k)\gamma$
for all $x\in Q$, $\textrm{Select}(Q,k)$ returns $k$ totally incomparable
sets each of size at least $\gamma/(k\log|Q|)$. 
\end{claim}

\begin{proof}
We prove this by induction on $k$. 

First, consider the case $k=2$. Note that $|T| \ge |Q|/2\ge 8\gamma > k$. When we
run $\textrm{Select}(Q,2)$, if $|D_{Q}(T)|\ge 2(2\lambda+\gamma)$, then by Claim
\ref{claim:condense}, we output two totally incomparable sets of size at
least $\gamma/(2\log|Q|)$. Otherwise, $|D_{Q}(T)|<2(2\lambda+\gamma) \le 4\gamma$, noting that $\lambda \le g(2)\gamma$ and $g(2)\le \frac{1}{2}$. In this case, the output is obtained from
$\textrm{Select}(T,1)$ and $\textrm{Select}(Q\setminus(T\cup D_{Q}(T)),1)$. Thus, we output two sets $T$ and $Q\setminus(T\cup D_{Q}(T))$ which are totally incomparable. Furthermore, $$|T| \ge |Q|/2 > \gamma/(2\log |Q|),$$ and $$|Q\setminus(T\cup D_{Q}(T))| \ge \lfloor|Q|/2\rfloor-4\gamma \ge 8\gamma - 1 - 4\gamma > \gamma/(2\log |Q|).$$Hence, the claim holds in the case $k=2$. 

Next, consider the case $k=3$. Similarly, when we run $\textrm{Select}(Q,3)$, noting that $|T| \ge |Q|/2 \ge \frac{1}{2}f(3)\gamma > k$, if $|D_{Q}(T)|\ge 2(2\lambda+\gamma)$, then by Claim
\ref{claim:condense}, we output three totally incomparable sets of size at
least $\gamma/(3\log|Q|)$. Otherwise, $|D_{Q}(T)|<2(2\lambda+\gamma) \le 4\gamma$, and the output is obtained from
$\textrm{Select}(T,2)$ and $\textrm{Select}(Q\setminus(T\cup D_{Q}(T)),1)$. Since $|T| \ge |Q|/2 \ge \frac{1}{2}f(3)\gamma \ge f(2)\gamma$, the claim in the case $k=2$ yields that $\textrm{Select}(T,2)$ outputs two totally incomparable sets each of size at least $\gamma/(2\log |Q|)$. Together with the set $Q\setminus(T\cup D_{Q}(T))$ which has size at least $\lfloor|Q|/2\rfloor-4\gamma > \gamma/(3\log|Q|)$, we obtain that in this case $\textrm{Select}(Q,3)$ outputs three totally incomparable sets each of size at least $\gamma/(3\log|Q|)$. Hence, the claim holds in the case $k=3$.
 
Assume that the claim is true for all $k'<k$ for some $k\ge 4$. By induction, we easily obtain that any function $f:\mathbb{Z}_+\to \mathbb{R}_+$ with $f(2)\ge 16$ and $f(k)\ge 2f(\lceil k/2\rceil)$ for all $k\ge 2$ satisfies that $f(k) \ge 8k$ for all $k\ge 2$. When we
run $\textrm{Select}(Q,k)$, noting that $|T| \ge |Q|/2 \ge \frac 1 2 f(k) \gamma \ge k$, if $|D_{Q}(T)|\ge 2(k\lambda+\gamma)$, then by Claim
\ref{claim:condense}, we output $k$ totally incomparable sets of size at
least $\gamma/(k\log|Q|)$. Otherwise, $|D_{Q}(T)|<2(k\lambda+\gamma)$, and the output is obtained from
$\textrm{Select}(T,\lceil k/2\rceil)$ and $\textrm{Select}(Q\setminus(T\cup D_{Q}(T)),\lfloor k/2\rfloor)$.
Note that 
\[
|T|\ge\frac{|Q|}{2}\ge \frac{1}{2} f(k) \gamma \ge f(\lceil k/2\rceil) \gamma,
\] 
and 
\begin{align*}
|Q\setminus(T\cup D_{Q}(T))| & \ge \left\lfloor\frac{|Q|}{2}\right\rfloor - 2(k\lambda+\gamma)\\
& \ge \frac{1}{2}f(k)\gamma - 1 - 2\gamma - 2kg(k)\gamma \\
& \ge \frac{1}{2}f(k)\gamma - 3\gamma - 2kg(k)\gamma \\
& \ge f(\lfloor k/2\rfloor)\gamma.
\end{align*}
Furthermore, since $\lfloor k/2\rfloor \le \lceil k/2\rceil \le k$ and $g$ is decreasing, we have that $\lambda \le g(k)\gamma \le g(\lceil k/2\rceil) \gamma  \le g(\lfloor k/2\rfloor) \gamma$. By induction, $\textrm{Select}(T,\lceil k/2\rceil)$ outputs $\lceil k/2\rceil$ pairwise totally incomparable sets and $\textrm{Select}(Q\setminus(T\cup D_{Q}(T)),\lfloor k/2\rfloor)$ outputs $\lfloor k/2\rfloor$ pairwise totally incomparable sets. The $k$ sets from the two outputs can be easily checked to be pairwise totally incomparable, and each of the set has size at least $\gamma/(k\log |Q|)$. This finishes the induction. 
\end{proof}

Lemma \ref{lem:incomp} readily follows from Claim \ref{claim:select}. 

\section{Multiple partial orders }

A sequence of sets $(A_{1},\dots,A_{k})$ is \textit{homogeneous} with respect
to the partial ordering $<$ if either all pairs of sets are totally incomparable,
or $A_{1}<A_{2}<\dots<A_{k}$, or $A_{1}>A_{2}>\dots>A_{k}$. In this
section, we prove the Theorem \ref{thm:multiple} showing that for
any list of partial orders on the same set, there exists $k$ large
subsets that are homogeneous with respect to each of the partial orders. 

We will use the following standard cake-cutting lemma in the proof of Theorem \ref{thm:multiple}. Cake-cutting and more generally fair division is a long-studied topic (see, for example, \cite{DS}). 

\begin{lem}\label{cakelemma}
Let $I \subset \mathbb{R}$ be an interval, $s$ be a positive integer, and $\mu_1,\ldots,\mu_s$ be absolutely continuous measures on $I$. Then there is a partition $I=I_1 \cupdot \dots \cupdot I_s$ into consecutive intervals and a permutation $\pi$ of $\{1,\ldots,s\}$ such that $\mu_{\pi(i)}(I_i) \geq \mu_{\pi(i)}(I)/s$ for $i=1,\ldots,s$. 
\end{lem}

The proof of Lemma \ref{cakelemma} follows from a greedy algorithm and induction on $t$. The base case $t=1$ is trivial. The greedy algorithm scans from the end of $I$, and if $I_t$ is the shortest ending interval for which there is a $j$ such that $\mu_j(I_t) \geq \mu_j(I)/t$, then we set $\pi(t)=j$. We then apply induction on the remaining interval $I \setminus I_t$ and the $t-1$ measures $\mu_i$ with $i \in [t] \setminus \{j\}$.   

We will need the following discrete consequence. 

\begin{cor}
Let $Q$ be a finite set with partition $Q=A_1 \cupdot \dots \cupdot A_k$. Let $B_1,\ldots,B_s$ be subsets of $Q$. Then there are integers $0=h_0 \leq h_1 \leq \cdots \leq h_s = h$ and a permutation $\pi$ of $\{1,\ldots,s\}$ such that $|B_{\pi(j)} \cap (\bigcup_{h_{j-1} < i\le h_{j}}A_{i})|\ge |B_{\pi(j)}|/s - \max_i |A_i|$ for $j =1,\ldots,s$.
\end{cor}
\begin{proof}
Let $I=[0,k]$. For $j=1,\ldots,s$, define the measure $\mu_j$ by $\mu_j([0,r])= |B_j \cap (\bigcup_{1 \leq i \leq r} A_i)|$ for each integer $r=0,\ldots,k$ and linearly interpolate between consecutive integers so that the measures $\mu_1,\ldots,\mu_s$ are absolutely continuous. Applying Lemma \ref{cakelemma}, there are intervals $I_j:=(r_{j-1},r_j]$ for $j=1,\ldots,s$ where $0=r_0 \leq r_1 \leq \ldots \leq r_s=k$, and a permutation $\pi$ of $\{1,\ldots,s\}$ such that $\mu_{\pi(j)}(I_j) \geq \mu_{\pi(j)}(I)/st$ for $j=1,\ldots,s$
. In order to discretize these intervals, we round to the next integer by letting $h_j=\lceil r_j \rceil$ for $j=1,\ldots,s$.  This choice of the $h_j$'s and $\pi$ has the desired property. 
\end{proof}

The following lemma is useful for the proof of Theorem \ref{thm:multiple}.

\begin{lem}
\label{lem:partition}Let $k>k' \geq 1$ be integers.  Suppose set $Q$ has a partition $Q=A_{1}\cupdot\dots\cupdot A_{k}$ into
subsets of equal size $a$. Let $B_{1},\dots,B_{k'}$ be disjoint
subsets of $Q$ of equal size $b \geq a$. Then there
exists $t_{1},t_{2},\dots,t_{k'/3}$ and $0=h_{0}<h_{1}<h_{2}<\dots<h_{k'/3}\le k$
such that $|B_{t_{j}}\cap(\bigcup_{h_{j-1} < i\le h_{j}}A_{i})|\ge\frac{b}{k'}$
for all $j\in[k'/3]$. 

\end{lem}

\begin{proof}
We consider the following iterative procedure. Let $r=b$ and $P=\bigcup_{i\le k'}B_{i}$.
Set $A_{i}^{0}=A_{i}\cap P$ for $i\le k$, and $F^{0}=\emptyset$. At step
$j \geq 0$,  let $U_{j,h}=\bigcup_{i \leq h} A_i^j$ and $h_j$ be minimum such that $|U_{j,h_j}| \geq r$. 
Then we let $t_{j}\in [k'] \setminus F^{j}$
be so that $B_{t_{j}}$ has the largest intersection with $U_{j,h_j}$.
We update $A_{i}^{j+1}=A_{i}^{j}\setminus(B_{t_{j}}\cup U_{j,h_j})$
for all $i\le k$ and update $F^{j+1}=F^{j}\cup\{t_{j}\}$. We stop
the process when $|U_{j,k}|<r$.  Observe that $A_{i}^{j+1}$ is disjoint from each $B_t$ with $t \in F^{j+1}$,  and $A_{i}^{j+1}$ is 
empty if $i \leq h_j$.  Notice that if the process does not stop by step $j$, then at step $j$,
we have $|B_{t_{j}}\cap U_{j,h_j}|\ge r/k'$ by the pigeonhole principle.  
Furthermore, since $A_{i}^{j}=\emptyset$ for all $i\le h_{j-1}$,
we have $|B_{t_{j}}\cap(\bigcup_{h_{j-1} < i\le h_{j}}A_{i})|=|B_{t_{j}}\cap U_{j,h_j}|\ge r/k'$.
Thus it suffices to show that there must be at least $k'/3$ steps
of the procedure before we stop.  Note that for each $j$ for which $U_{j+1,k}$ is defined,  we have $|U_{j,k}|-|U_{j+1,k}| \leq r+a+b \leq 3b$.  As $|P| =k'b$,  we can continue for at least $k'/3$ steps,  as desired. 
\end{proof}

We next give the proof of Theorem \ref{thm:multiple}.
\begin{proof}[Proof of Theorem \ref{thm:multiple}]
Define $k_h=k$ and inductively define $k_{\ell-1}=(10k_\ell)^{12}(\log n)$
for $\ell\le h$. Let $n_1=\frac{n}{10^{4}k_1^{2}(\log n)}$,
and inductively, $n_{i+1}=\frac{k_in_i}{(10k_{i+1})^{12}(\log n)}$
for $1\le i\le h-1$. Note that $n_{i+1}\ge n_{i}$ for $1\le i\le h-1$
for sufficiently large $n$. We prove by induction on $\ell\le h$
that for $n$ sufficiently large, we can find a sequence of $k_\ell$ sets which is homogeneous with respect to $<_{1},\dots,<_{\ell}$, and each set has size $n_\ell$.
This is true for $\ell=1$ by Theorem \ref{thm:multi-dilworth}.  Assume
that the claim holds for all $\ell'<\ell$; we prove the claim for
$\ell$. 

Let $(A_{1},\dots,A_{k_{\ell-1}})$ be a homogeneous sequence of sets with respect
to $<_{1},\dots,<_{\ell-1}$, where each set in the sequence has size $n_{\ell-1}$. Let $Q=A_{1}\cup\dots\cup A_{k_{\ell-1}}$.
Let $k'_\ell=3k_\ell^{2}$. By Theorem \ref{thm:multi-dilworth},
we can find subsets $B_{1},\dots,B_{k'_\ell}$ of $Q$ each of size
$\frac{|Q|}{40(k'_{\ell})^{2}(\log|Q|)}$ which are pairwise
totally incomparable with respect to $<_{\ell}$, or we can find subsets $B_{1},\dots,B_{k'_{\ell}}$
of $Q$ each of size $\frac{|Q|}{10^{4}(k'_\ell)^{5}}$
and $B_{1}<_{\ell}\dots<_{\ell}B_{k'_\ell}$ or $B_{1}>_{\ell}\dots>_{\ell}B_{k'_\ell}$.
Let $B=B_{1}\cup\dots\cup B_{k'_\ell}$. Note that 
\[
\min\left(\frac{|Q|}{40(k'_\ell)^{2}(\log|Q|)},\frac{|Q|}{10^{4}(k'_\ell)^{5}}\right)\ge n_{\ell}.
\]

First, consider the case $B_{t},B_{t'}$ are totally incomparable for all
$t\ne t'$. By Lemma \ref{lem:partition}, there exists there exists
$t_{1},t_{2},\dots,t_{k'_\ell/3}$ and $0=h_{0}<h_{1}<h_{2}<\dots<h_{k'_\ell/3}\le k_{\ell-1}$
such that for all $j\in[k'_\ell/3]$, $$\left|B_{t_{j}}\cap\left(\bigcup_{h_{j-1} < i\le h_{j}}A_{i}\right)\right|\ge\frac{1}{k'_\ell}\cdot\frac{|Q|}{40(k'_\ell)^{2}(\log|Q|)}.$$The sets $B_{t_{j}}\cap(\bigcup_{h_{j-1} < i\le h_{j}}A_{i})$
for $j\in[k'_\ell/3]$ form a homogeneous sequence of sets with respect to $<_{1},\dots,<_{\ell-1},<_{\ell}$,
and each set has size at least $\frac{k_{\ell-1}n_{\ell-1}}{40(k'_\ell)^{3}(\log n)}\ge n_\ell$.
Since $k'_\ell/3>k_\ell$, we obtain the desired conclusion in this
case. 

Next, consider the case $B_{1}<_{\ell}B_{2}<_{\ell}\dots<_{\ell}B_{k'_\ell}$
(the case $B_{1}>_{\ell}B_{2}>_{\ell}\dots>_{\ell}B_{k'_\ell}$ can
be treated similarly). By Lemma \ref{lem:partition}, there exists
there exists $t_{1},t_{2},\dots,t_{k'_\ell/3}$ and $0=h_{0}<h_{1}<h_{2}<\dots<h_{k'_\ell/3}\le k_{\ell-1}$
such that for all $j\in[k'_\ell/3]$, $$\left|B_{t_{j}}\cap\left(\bigcup_{h_{j-1} < i\le h_{j}}A_{i}\right)\right|\ge\frac{1}{k'_\ell}\cdot\frac{|Q|}{10^{4}(k'_\ell)^{5}}.$$Let $C_{j}=B_{t_{j}}\cap(\bigcup_{h_{j-1} < i\le h_{j}}A_{i})$,
then for $j\ne j'$, either $C_{j}>_{\ell}C_{j'}$ or $C_{j}<_{\ell}C_{j'}$,
and furthermore $(C_{1},C_{2},\dots,C_{k'_\ell/3})$ is a homogeneous sequence of sets with respect to $<_{1},\dots,<_{\ell-1}$. By Erd\H{o}s-Szekeres theorem, we can find
$\sqrt{k'_\ell/3}=k_\ell$ indices $j_{1},\dots,j_{k_\ell}$ such
that $j_{1}<\dots<j_{k_\ell}$ and either $C_{j_{1}}<_{\ell}\dots<_{\ell}C_{j_{k_\ell}}$
or $C_{j_{1}}>_{\ell}\dots>_{\ell}C_{j_{k_\ell}}$. Thus, $(C_{j_{1}},\dots,C_{j_{k_\ell}})$
forms a homogeneous sequence of sets with respect to $<_{1},\dots,<_{\ell-1},<_{\ell}$,
where each set in the sequence has size at least $\frac{1}{k'_\ell}\cdot\frac{|Q|}{10^{4}(k'_\ell)^{5}}=\frac{k_{\ell-1}n_{\ell-1}}{10^{4}(k'_\ell)^{6}}\ge n_\ell$,
completing the induction. 

Thus, we can find a homogeneous sequence of sets $(A_{1},\dots,A_{k})$ with
respect to $<_{1},\dots,<_{h}$ such that each set in the sequence has size at least 
\[
n_h\ge n_1\ge\frac{n}{10^{4}(10k\log n)^{2(1+12+12^{2}+\cdots+12^{h})}}\ge\frac{n}{(10k\log n)^{12^{h+1}}}. \qedhere
\]
\end{proof}

\section*{Acknowledgements}
We would like to thank Istv\'an Tomon for helpful comments and for pointing out the application of Theorem \ref{thm:multiple} to Erd\H{o}s-Hajnal type questions in matrices.

\end{document}